\documentclass{amsart}
\pdfoutput=1
\usepackage[utf8]{inputenc}
\usepackage[pdftex]{graphicx}
\usepackage{subfigure}
\usepackage{color}
\usepackage{amsmath,amssymb,amsfonts,amsthm}
\usepackage{amscd,enumerate}
\usepackage{latexsym}
\usepackage{euscript}
\usepackage{setspace}
\usepackage{enumitem}
\usepackage{caption}

\newtheorem{thmm}{Theorem}

\newtheorem{thm}{Theorem}
\newtheorem{prop}[thm]{Proposition}
\newtheorem{lemma}[thm]{Lemma}

\newtheorem{condition}[thm]{Conditions}

\theoremstyle{definition}
\newtheorem{defi}[thm]{Definition}

\newtheorem{exem}[thm]{Example}

\numberwithin{thm}{section} % Numeração de teoremas do tipo 1.1, A.B

\theoremstyle{remark}
\newtheorem*{remark}{Remark}

\newcommand{\R}{\mathbb{R}}

\newcommand{\N}{\mathbb{N}}

\newcommand{\W}{\mathcal{W}}

\newcommand{\pp}{\mathcal{P}}
\newcommand{\G}{\mathcal{G}}
\newcommand{\Ss}{\mathcal{S}}

\newcommand{\diam}{\operatorname{diam}}
\newcommand{\dis}{\operatorname{dis}}

\newcommand{\dgh}{d_{GH}} %gromov-hausdorff

\newcommand{\talpha}{\tilde{\alpha}}

\begin{document}

\title[Gromov-Hausdorff Convergence]{Gromov-Hausdorff Convergence of Metric Quotients and Singular Conic-Flat Surfaces}
\author{Marcel Vinhas}
\date{July 2020}
% Talvez tenha que converter para 2020.
% 57M50 General geometric structures on low-dimensional manifolds
% 51F99 Metric geometry: None of the above, but in this section
% 57K20 2-dimensional topology (including mapping class groups of surfaces, Teichmüller theory, curve complexes, etc.)
% 53C45 Global surface theory (convex surfaces à la A. D.Aleksandrov)
% 53C23 Global geometric and topological methods (à la Gromov); differential geometric analysis on metricspaces

\begin{abstract}
Given metric quotients $S$ and $S_n$, $n \in \N$, of a metric space $X$, sufficient conditions are provided on the data defining them guaranteeing that $S$ is the Gromov-Hausdorff limit of $S_n$.
These conditions are recognized within metric quotients of plane polygons determined by side-pairings known as plain paper-folding schemes. 
In particular, concrete examples are given of sequences of two-dimensional conic-flat spheres converging to spheres that are conic-flat except around certain singularities, some of them with unbounded curvature in the sense of comparative geometry.
\end{abstract}

\maketitle

\section*{Introduction}

For a fixed metric space $X$, consider the family $\Ss$ of every metric quotient $S$ obtained from $X$.
The first result in this paper is a tool for proving that a given $S \in \Ss$ is the limit of a given sequence $S_n \in \Ss$. 
It is based on comparing the data $\G$ and $\G_n$ that produces the respective quotients, giving precision to the idea that if these data ``look alike'', then the associated metric quotients are close with respect to the the Gromov-Hausdorff distance.
Intuitively, the conditions of Theorem \ref{thm:main} require that for a ``collapsing'' sequence $D_n$ of subsets of $X$, points outside $D_n$ are equally related by $\G$ and $\G_n$, and points outside $D_n$ are related to points inside it by $\G$ and $\G_n$ in a controlled manner.
More precisely, for a collection $\G$ of subsets of $X$ and $x,y \in X$, denote $x \, \G \, y$ if $x = y$ or there exist $g \in \G$ with $x,y \in g$. 
The boundary, the complement and the diameter of each subset $D$ of $X$ are denoted, respectively, $\partial D$, $X \setminus D$ and $\diam D$; and the Gromov-Hausdorff distance between metric spaces $S$ and $S'$ is denoted $\dgh (S,S')$.

\begin{thmm}
\label{thm:main}
Let $S$ and $S_n$, $n \in \N$, be the metric quotients associated to collections $\G$ and $\G_n$ of subsets of a metric space $X$. 
Suppose that, for each $n \in \N$, there exist $g^n \in \G$, $g_n \in \G_n$ and $D_n \subset X$ such that:
\begin{enumerate}[label=\roman*)]
 \item \label{thm:gh2-hyp1}
 For any $x, y \in X \setminus D_n$, $x \, \G \, y$ if, and only if, $x \, \G_n \, y$.
 
 \item \label{thm:gh2-hyp2}
 For any $x \in X \setminus D_n$ and $y \in D_n$, $x \, \G \, y$ implies that $x, y \in g^n$, and $x \, \G_n \, y$ implies that $x,y \in g_n$.
 
 \item \label{thm:gh2-hyp3}
 $g^n \cap (\partial D_n \setminus D_n) \neq \emptyset$ and $g_n \cap (\partial D_n \setminus D_n) \neq \emptyset$. 
\end{enumerate}
If $\lim_{n \to + \infty} \diam D_n = 0$, then $\lim_{n \to + \infty} \dgh (S,S') = 0$.
\end{thmm}

In case $X$ is compact, standard results on the Gromov-Hausdorff topology (see Theorem 7.4.15 in \cite{BBI}) guarantee that any sequence $S_n \in \Ss$ contains a subsequence that converges to a compact metric space.
In this context, Theorem \ref{thm:main} may be useful for recognizing such limit, and for constructing sequences approximating a given limiting space. Taking a $X$ as fixed polygon $P$ in some model plane, collections $\G$ associated to side-pairings $\pp$ of $P$ known as paper-folding schemes \cite{dCH} are considered.
These pairings glue together, isometrically, interior-disjoint plane segments contained in $\partial P$, as in classical surface theory.
However, infinitely many pairings are allowed, provided that the paired segments cover $\partial P$ up to measure-zero.
While the formal definitions are a bit lengthy (section \ref{sec:back-paperspaces}), it is easy to come up with and represent the simplest identifications patterns in paper-folding schemes as in Figures \ref{fig:general}, \ref{fig:canon1} and \ref{fig:repeat}, where dotted lines connect paired points.
Specifically, plain paper-folding schemes are considered, this being a restriction on how the paired points are linked along $\partial P$.
These are known to produce quotients homeomorphic to the $2$-sphere.
A particular application of the following result is shown in Figure \ref{fig:canon1}.

\begin{thmm}[Theorem \ref{thm:paperlimit}]
\label{thm:main2}
Given a plain paper-folding scheme, sequences of plain paper-folding schemes approximating it in the Gromov-Hausdorff sense are constructed.
\end{thmm}

If $\pp$ prescribe only finitely many pairings, the associated quotient $S$ is a conic-flat surface, also known as polyhedral, in the sense that every point has either a flat or conical neighborhood.   
On the other hand, gluing by infinitely many pairings may produce points in $S$ around which the metric is not described by these models.
These are called \textit{singular}, and simple examples of such are accumulations of conical points, and points with infinite total angle around it.
Theorem \ref{thm:main2} is applied to approximate examples of conic-flat spheres with singularities, in the Gromov-Hausdorff sense, by conic-flat spheres without singularities (Examples \ref{exem:general}, \ref{exem:canon1} and \ref{exem:repeat}).
This is known to imply that the convergence also occurs uniformly -- see Exercise 7.15.14 in \cite{BBI}.
Some singularities in the given examples implies that $S$ is not a space of bounded curvature in the sense of comparative geometry \cite{BBI}, both above and below.
This comes together with the total curvature exploding along the approximating sequence, in accordance with the theory of surfaces of bounded curvature \cite{AZ}, which guarantee that if certain bounds on the curvature of conical points hold for a sequence, then the property of being a surface of bounded curvature passes to uniform limits.

Theorems \ref{thm:main} and \ref{thm:main2} outgrew from particular cases first established in \cite{master}.
Originally, paper-folding schemes were considered due to their associated quotients being the domains of certain surface homeomorphisms that are relevant in the theory of dynamical systems (see \cite{unimodal, dCH} and references therein), later gaining attention also in three-manifold geometry/topology (recent developments are found in \cite{palimits}). 
In these contexts, both the quotients and the transformations show up in families.
Besides the geometric content mentioned above, the results in this paper were also motivated by the problem of passing limits along these families, which is tackled in \cite{dCH,dCH2} by means of uniformization techniques of complex analysis.
While the results presented here are more restrictive, due to requiring a fixed polygon and leaving aside any transformations defined on the quotients, the author hopes that they consist the first step in an alternate approach to this matter.

The paper is structured as follows: section \ref{sec:back} summarizes definitions and results on metric quotients, the Gromov-Hausdorff distance between compact metric spaces, and paper-folding schemes. The proofs of Theorems \ref{thm:main} and \ref{thm:main2} are given in sections \ref{sec:gromovhausdorff} and \ref{sec:paperlimit}, the later also containing the aforementioned examples. 
The author thanks to Andr\'e de Carvalho for presenting him to paper-folding schemes, and aknowledges the partial finantial support of CNPq, FAPESP and UFPA for his research.

\section{Background}
\label{sec:back}

This section summarizes the basic concepts developed in the main results of the paper.
Metric spaces, quotients and the Gromov-Hausdorff distance are presented as in \cite{BBI}, with minor non-essential modifications (the terminology adopted in the definition of metric quotients follows \cite{bonahon}). 
For paper-folding schemes, the main reference is \cite{dCH}.

\subsection{Metric quotients and intrinsic metrics} 
\label{sec:metquo}

For a non-empty set $X$, a function $d : X \times X \to \R \cup \{ \infty \}$ is a \textit{semi-metric} if, for any $x, y, z \in X$, $d(x,x) = 0$; $d(x,y) = d(y,x)$; and the triangle inequality is valid: $d(x,z) \leq d(x,y) + d (y,z)$.
A \textit{metric} on $X$ is a semi-metric on $X$ such that $d(x,y) = 0$ implies that $x = y$, and, in this case, the pair $(X,d)$ is a \textit{metric space}.
For a semi-metric $d$ on $X$, $d(x,y) = 0$ defines an equivalence relation $\sim$ on $X$, and $d$ induces a metric, also denoted $d$, in the quotient $X / {\sim}$.

Each collection $\G$ of subsets of $X$ induces a reflexive and symmetric relation on $X$, defined by $x \, \G \, y$ if, and only if, $x = y$ or there exists $g \in \G$ with $x, y \in g$.

\begin{defi} 
\label{defi:metquo}
Let $(X,d)$ be a metric space and $\G$ be any collection of subsets of $X$.
For $x, y \in X$, a \emph{$\G$-walk} $\W$ from $x$ to $y$ is a finite sequence of pairs $\{x_j, y_j \}_{j=1}^N$ of points in $X$ such that, for each $1 \leq j \leq N -1$, $y_j \, \G \, x_{j+1}$. 
Each $\{ x_j, y_j \}$ is a \textit{step} of $\W$ and each $\{ y_j \, ; \, x_{j+1}\}$ is a \textit{jump} of $\W$. 
The $\G$-walk $\{ x, y\}$ is called \emph{trivial}.
The \textit{length} of a step $\{ x_j, y_j \}$ is equal to $d(x_j,y_j)$, and the \emph{length} of $\W$ is equal to $\sum_{j = 1}^N d (x_j, y_j)$.
For any $x,y \in X$, 
\begin{equation}
d^\G (x,y) = \inf \sum_{j = 1}^N d (x_j, y_j) \, ,
\end{equation}
the infimum being over every $\G$-walk from $x$ to $y$, is the \textit{quotient semi-metric of $X$ associated to $\G$}.
The equivalence relation on $X$ defined by $d^\G (x,y) = 0$ is denoted $\sim_\G$, and the induced metric on the quotient $X/{\sim}_\G$ is also denoted $d^\G$.
The \emph{metric quotient associated to $\G$} is the metric space $(X/{\sim}_\G, d^\G)$.
The \textit{projection map $\pi: X \to X/{\sim}_\G$} associates to each $x$ its ${\sim}_\G$-equivalence class.
\end{defi}

Due to the trivial walk, $d^\G (x,y) \leq d (x,y) $ for every $x,y \in X $.
So, the projection map does not increase distances:
\begin{equation}
 d^\G (\pi (x), \pi(y)) \leq d (x,y) \, .
\end{equation}
In particular, it is a continuous map from $(X,d)$ to $(X/{\sim}_\G, d^\G)$, of course also surjective.
% Of course $\pi$ is surjective, and it follows that if $X$ is compact / connected, then so it is $X / d^\G$.
% In the first case, $\pi$ is also closed, as $X/d^\G$ is a metric space, and therefore Hausdorff.
% For each curve $\gamma$ in $X$, $\pi \circ \gamma$ is a curve in $X / d^\G$ whose length is at most the length of $\gamma$.
% It is not hard to see that $X/d^\G$ is a length space if $X$ is length space (concatenate curves whose lengths approximate the distances between the points of the steps of each $\G$-walk).
% In particular, Theorem \ref{thm:compactlengthspace} guarantee that $d^\G$ is strictly intrinsic if $X$ is a compact length space.

The collection $\G$ can be arbitrary, and is not assumed to be a decomposition of $X$, as in the definition of the quotient topology. This makes easier to describe $\G$, and anyway is not essential, since decompositions of $X$ generating the same quotient semi-metric can be obtained from any collection of subsets of $X$.
Of course, if $x,y \in g$ for some $g \in \G$, then $d^\G (x,y) = 0$. 
But, in general, the quotient topology and the metric quotient do not coincide, as there are other situations when $d^\G (x,y) = 0$. 
For instance, if $y$ is an accumulation point of $g \in \G$, then $d^\G (x,y) = 0$ for every $x \in g$. 
In this case, if $y \notin g$, the quotient topology fails to be Hausdorff, and so is not even metrizable. 
This happens frequently among collections of subsets associated to paper-folding schemes. 
Furthermore, in general, the topology of the metric quotient $(X/{\sim_\G},d^\G)$ may not be equivalent to the quotient topology on $X/{\sim}_\G$. 
However, if $X$ is compact, the identity map of $X/{\sim}_\G$ is a homeomorphism between these topologies, since it is a continuous bijection from a compact space to a Hausdorff one.

\begin{defi}
The \textit{diameter} of a subset $D$ of a metric space $(X,d)$ is defined by $\diam D = \sup_{x,y \in D} d (x,y)$.
\end{defi}

\begin{defi}
Let $\gamma : [a,b] \to X$ be a curve in a metric space $(X,d)$. 
The \textit{length of $\gamma$ in $X$} is: 
\begin{equation}
|\gamma| = \sup \sum_{i=0}^{N-1} d( \gamma(t_i) , \gamma(t_{i+1}) ) \in \R_{\geq 0} \cup \{ \infty \},
\end{equation}
where the supremum is taken over every partition $a = t_0 < \cdots < t_N = b$. 

For each subset $A$ of a metric space $(X,d)$, the induced \emph{intrinsic metric of $A$} is defined, for any $x, y \in A$, by:
\begin{equation}
d_A (x,y) = \inf |\gamma| \, , 
\end{equation}
where the infimum is over every curve $\gamma$ contained in $A$ connecting $x$ and $y$. 
In particular, $d_A (x,y) = \infty$ if there is no such curve. 
The metric $d$ is \emph{intrinsic} if $d_X = d$ and, in this case, $(X,d)$ is a \emph{length space}. Also, $d$ is \emph{strictly intrinsic} if, for any points at finite distance from each other, the infimum of the definition is realized by some path connecting them.
\end{defi}

\begin{remark}
 \begin{enumerate}
  \item 
  If a length space is compact, then its metric is strictly intrinsic.
  
  \item
  If $X$ is a length space, then any metric quotient of $X$ is a length space.
  In particular, the metric of any metric quotient of a compact length space is strictly intrinsic.
 \end{enumerate}
\end{remark}

\subsection{Gromov-Hausdorff distance}
\label{sec:back-gh}
The way of defining the Gromov-Hausdorff distance that better suits the purposes of the present paper is based on the concept of correspondences between metric spaces.

\begin{defi}
Let $(X_1, d_1)$ and $(X_2,d_2)$ be metric spaces. A \textit{correspondence between $X_1$ and $X_2$} is a subset $R \subset X_1 \times X_2$ with the following properties: for each $x_1 \in X_1$, there exist $x_2 \in X_2$ such that $(x_1,x_2) \in R$; and for each $x_2 \in X$, there exist $x_1 \in X$ such that $(x,y) \in R$. In both cases, uniqueness is not required. The \textit{distortion of a correspondence $R$} is:
\begin{equation}
 \dis R = \sup \{| d_1 (x_1, x_1') - d_2 (x_2, x_2') | \, : \, (x_1, x_2), (x_1 ', x_2 ') \in R \}.
\end{equation}
% Notice that $\dis R \leq \diam X_1 + \diam X_2$ \texttt{É utilizado?}.
\end{defi}

To each correspondence $R \subset X_1 \times X_2$, the projections $R \to X_1$ and $R \to X_2$ are surjective. Reciprocally, given a pair of surjective functions $f_i : X \to X_i$, $R = \{ (f_1 (x), f_2 (x)) \in X_1 \times X_2 \, : \, x \in X \}$ is a correspondence between $X_1$ and $X_2$. Its distortion is given by:
\begin{equation}
 \dis R = \sup_{x,x' \in X} |d_1 (f_1 (x), f_1(x')) - d_2 (f_2 (x), f_2 (x')) | .
\end{equation}

\begin{defi}
Let $X_1$ and $X_2$ be metric spaces. The \textit{Gromov-Hausdorff distance between $X_1$ and $X_2$} is defined by:
\begin{equation} \label{eqn:dghdis}
 \dgh (X_1,X_2) = \frac{1}{2} \inf_R \dis R \, ,
\end{equation}
where the infimum is over every correspondence $R \subset X_1 \times X_2$.

\end{defi}

\begin{defi}
 Let $(X,d)$ be a metric space. 
 For each $r > 0$, an \emph{$r$-net on $X$} is a subset $A \subset X$ with the property that, given $x \in X$, there exist $a \in A$ such that $d(x,a) < r$. 
\end{defi}

\begin{remark}
 \begin{enumerate}
  \item
  The Gromov-Hausdorff distance satisfies the triangle inequality.
  
  \item 
  If $A$ is an $r$-net on a metric space $(X,d)$ and is considered as a metric space with the restriction of $d$, then $\dgh (X,A) \leq r$.
 \end{enumerate}

\end{remark}

\subsection{Paper-folding schemes}
\label{sec:back-paperspaces}

In the sequence, ``the plane'' is a fixed model plane. 
Namelly, either a hyperbolic plane, or the Euclidean plane, or a round $2$-sphere. 
Originally, paper-folding schemes and paper spaces were defined in \cite{dCH} only in the Euclidean setting.
This is imaterial for the present purposes, and the definitions therein work in the more general context with only small adaptations. 

\begin{defi}
An \textit{arc} in a metric space $X$ is a homeomorphic image $\gamma \subset X$ of the interval $[0,1]$. 
Its \textit{endpoints}, or \textit{extremities}, are the images of $0$ and $1$, and its \textit{interior} is the image $\overset{\circ}{\gamma}$ of the open interval $(0,1)$. 
A \textit{segment} is an arc in the plane that is a subset of a geodesic. 
The length of a segment $\alpha$ is denoted $|\alpha|$.
A \textit{simple closed curve} in $X$ is a homeomorphic image of the unit circle.
An arc or simple closed curve in the plane is \textit{polygonal} if it is the concatenation of finitely many segments. 
Its \textit{vertices} are the intersections of consecutive maximal segments, and the maximal segments themselves are its \textit{edges}. 
A (polygonal) multicurve is a disjoint union of finitely many (polygonal) simple closed curves.
A \textit{polygon} is a closed topological disk in a plane whose boundary is a polygonal simple closed curve. 
Its \textit{vertices} are the vertices of its boundary, and its \textit{sides} are the edges of its boundary.
In the spherical case, a polygon is assumed to be properly contained in a hemisphere.
A \textit{plane multipolygon} is a disjoint union of finitely many plane polygons, which may belong to distinct copies of a same plane. 
The intrinsic metric of a multipolygon $P$ induced by the ambient plane(s) metric is denoted $d_P$. 
The boundary of a plane multipolygon will be considered with its positive orientation induced by the orientation of the plane.
\end{defi}

\begin{defi}
Let $C$ be an oriented polygonal multicurve and $\talpha, \talpha' \subset C$ be segments of with the same length and disjoint interiors. 
The associated \textit{segment pairing} $\langle \talpha, \talpha ' \rangle$ is the decomposition of $\talpha \cup \talpha '$ obtained by gluing $\talpha$ and $\talpha '$ isometricaly reversing its orientations. 
More precisely, for parametrizations of $\talpha$ and $\talpha '$ by arc length compatible with their orientations, each element of $\langle \talpha, \talpha ' \rangle$ is of the form $\{ \talpha (t) , \talpha ' (|\talpha '| - t) \}$, and these points are said \textit{paired}. 
Paired points that belong to the interior of the paired segments constitutes \textit{interior pairs}. 
A \textit{fold} is a pairing whose segments have a common endpoint, this point being called a \textit{folding point}. 
The \textit{length} of a pairing is defined by $|\langle \talpha, \talpha ' \rangle| = |\talpha | = |\talpha '|$. A collection $\pp = \{ \langle \talpha_i, \talpha_i ' \rangle\}_i$ is \textit{interior disjoint} if the interiors of all the segments $\talpha_i$ and $\talpha_j '$ are disjoint. 
Notice that, in this case, $\pp$ is at most countable. 
It is \textit{full} if $\sum_i |\langle \alpha_i , \alpha_i ' \rangle|$ is equal to half the length of $C$. 
For a interior disjoint collection $\pp$ of segment pairings, $\pp$ will also denote the collection of subsets of $C$ whose elements are points paired by some element of $\pp$, and the associated reflexive and symmetric relation on $C$.
\end{defi}

\begin{defi}
A \textit{paper-folding scheme} is a pair $(P, \pp)$, where $P$ is multipolygon with its intrinsic metric $d_P$, and $\pp = \{ \langle \alpha_i , \alpha_i ' \rangle \}_i$ is full interior disjoint collection of segment pairings of $\partial P$. 
% Let $d_P^\pp$ be the associated quotient semi-metric. 
The metric quotient $(S,d_S) = (P / {\sim_\pp}, d_P^\pp)$ is the associated \textit{paper space}. 
Recall that the projection map is denoted $\pi : P \to S$.
% The \textit{scar} of $S$ is $G = \pi(\partial P) \subset S$.
\end{defi}

\begin{remark}
For a paper-folding scheme $(P,\pp)$, each interior point of $P$ is $\sim_\pp$-equivalent only to itself.
In fact, the restriction of the projection map to the interior of $P$ is a homeomorphism onto its image.
Also, each interior pair $\{ z, z' \}$ is precisely the $\sim_\pp$-equivalence class of it points, while each folding point coincides with its $\sim_\pp$-equivalence class. 
Every paper space is a compact length space, homeomorphic to the topological quotient $P/{\sim_\pp}$.
Its metric is strictly intrinsic and, away from a singular set, locally isometric to metric cones on circles.
In particular, this singular set is empty if $\pp$ consists of finitely many pairings and, in this case, $S$ is a conic-flat surface without singularities, also known as a polyhedral surface. 
For more details on this, as well as topological, measure-theoretic, and conformal developments of the subject, see \cite{dCH,dCH2}.
More information on the geometry of a paper space around certain kinds of singularities will be given in further works.
\end{remark}

This paper giver particular attention to paper-folding schemes called plain, that will now be defined.
Theorem \ref{thm:plain} below is contained in Lemmas 38 and 41, and Theorem 42, in \cite{dCH}.

% \begin{prop} \label{prop:paper-basic}
%  The projection map $\pi : P \to S$ associated to a paper-folding scheme $(P, \pp)$ does not increases distances. In particular, $\pi$ is a continuous surjection. The associated paper space $S$ is a compact length space, homeomorphic to the topological quotient $P/{\sim_\pp}$, and its metric is strictly intrinsic. The scar $G$ is a compact subspace of $S$.
% \end{prop}

% Definitions 34 in \cite{dCH}
\begin{defi}
\label{defi:plainscheme}
 Let $\gamma$ be a polygonal arc or polygonal simple closed curve. 
 Two pairs of (not necessarily distinct) points $\{ x, x ' \}$ and $\{ y, y ' \}$ of $\gamma$ are \textit{unlinked} if one pair is contained in the closure of a connected component of the complement of the other. 
 Otherwise, they are \textit{linked}.
 A reflexive and symmetric relation $R$ on $\gamma$ is \textit{unlinked} if any two unrelated pairs of related points are unlinked: that is, if $x \, R \, x'$, $y \, R \, y'$, and neither $x$ nor $x'$ is related to either $y$ or $y '$, then $\{ x, x ' \}$ and $\{ y, y ' \}$ are unlinked. 
 An interior disjoint collection $\pp$ of segment pairings on $\gamma$ is unlinked if the corresponding relation $\pp$ is unlinked.
 A paper-folding scheme $(P, \pp)$ is \textit{plain} if $P$ is a single polygon and $\pp$ is unlinked. 
\end{defi}

\begin{defi}
 Let $R$ be a reflexive and symmetric relation on a set $X$.
 A subset $U$ of $X$ is $R$-saturated if it contains $\{ y \in X \, | \, y \, R \, x \}$ for every $x \in U$.
\end{defi}

% Definitions 36 in \cite{dCH}
\begin{defi}
 \label{defi:plainarc}
 Let $(P,\pp)$ be a paper-folding scheme.
 An arc $\gamma \subset \partial P$ is $\pp$-\textit{plain} if:
 \begin{enumerate}
  \item 
  Every pairing in $\pp$ which intersects the interior of $\gamma$ is contained in $\gamma$ (that is, if $\langle \alpha, \alpha ' \rangle$ is a segment pairing and either $\alpha$ or $\alpha '$ intersects the interior of $\gamma$, then both $\alpha$ and $\alpha '$ are contained in $\gamma$); and
  \item
  The restriction of $\pp$ to $\gamma$ is unlinked.
 \end{enumerate}
A component $\gamma$ of $\partial P$ is \textit{plain} if it is $\pp$-saturated and the restriction of $\pp$ to $\gamma$ is unlinked. 
In particular, a paper folding scheme $(P, \pp)$ is plain if $P$ consists of a single polygon whose boundary $\partial P$ is plain.
\end{defi}

\begin{thm}[\cite{dCH}]
\label{thm:plain}
Let $(P, \pp)$ be a paper-folding scheme, $\sim_\pp$ be the equivalence relation induced by the quotient semi-metric $d^\pp$, and $\gamma \subset \partial P$ be an arc with endpoints $a$ and $b$.
\begin{enumerate}
 \item
 If $\gamma$ is $\pp$-plain, then $a \sim_\pp b$.
 
 \item 
 $\gamma \setminus [a] = \gamma \setminus [b]$ is $\sim_\pp$-saturated.
 
 \item
 Suppose that $(P, \pp)$ is plain and let $x,y \in \partial P$ be distinct points which are not in interior $\pp$-pairs.
 Then $x \sim_\pp y$ if, and only if, an arc in $\partial P$ with endpoints $x$ and $y$ is plain.
 
 \item
 If $(P, \pp)$ is plain, then the associated paper space is homeomorphic to the two-dimensional sphere.
\end{enumerate}
\end{thm}

\section{Proof of Theorem \ref{thm:main}}
\label{sec:gromovhausdorff}

The proof is divided into Lemmas \ref{thm:gh1} and \ref{lemma:gh2-2} on pairs of metric quotients of $X$. 
The first relates the Gromov-Hausdorff distance between them and the difference between their associated semi-metrics over nets on $X$, while the second estimates this difference in the complement of a set $D$ as in the statement of Theorem \ref{thm:main}.
It will be convenient to reformulate its hypothesis as Conditions \ref{thm:gh2-hyp} relating pairs of collections of subsets on $X$.
The proof of Lemma \ref{lemma:gh2-2} depends on a simple general result on metric spaces, stated and proved as Proposition \ref{prop:gh2-1}.

\begin{lemma} 
\label{thm:gh1}
 Let $(X,d)$ be a metric space and $\G_i$ be collections of subsets of $X$, $i = 1, 2$. 
 Consider the associated semi-metrics $d^{\G_i}$ on $X$. 
 Given $r > 0$, if there exist an $r$-net $A$ on $X$ such that 
 \begin{equation} \label{eqn:thmgh1-hyp}
\sup_{a,a' \in A} |d^{\G_1} (a,a') - d^{\G_2} (a,a')| \leq r \, ,
 \end{equation}
 then $\dgh (S_1, S_2) < 5 r / 2$.
\end{lemma}
\begin{proof}
Denote by $d_i$ the metric of the metric quotients $S_i$ and by $\pi_i : X \to S_i$ the projection maps. 
Since $A$ is an $r$-net on $X$, each $A_i = \pi_i (A)$ is an $r$-net on $S_i$. 
Therefore, $\dgh(S_i, A_i) < r$, where $A_i$ is considered as a metric space with the restriction of $d_i$. 
This, with the triangle inequality, gives:
\begin{eqnarray} 
 \dgh (S_1, S_2) & \leq & \dgh (S_1, A_1) + \dgh (A_1, A_2) + \dgh (S_2, A_2) \\
                 & <    & 2 r + \dgh (A_1, A_2) \, . \label{eqn:thmgh1-1}
\end{eqnarray}

To bound the last term, it suffices to bound the distortion of some correspondence $R$ between $A_1$ and $A_2$, due to formula (\ref{eqn:dghdis}) for the Gromov-Hausdorff distance. 
So let $R$ be induced by restrictions $\pi_1 : A \to A_1$ and $\pi_2 : A \to A_2$. 
Recall that, by the definition of quotient metric, $d_i (\pi_i (x), \pi_i (y)) = d^{\G_i} (x,y)$ for any $x,y \in X$. 
Then, the hypothesis (\ref{eqn:thmgh1-hyp}) bounds the distortion of $R$:
\begin{eqnarray}
 \dgh (A_1, A_2) & \leq & \frac{1}{2} \dis R \\
                 & =    & \frac{1}{2} \sup_{a,a' \in A} | d_1 (\pi_1 (a), \pi_1 (a')) - d_2 (\pi_2 (a), \pi_2 (a')) |  \\
		 & = 	& \frac{1}{2} \sup_{a,a' \in A} |d^{\G_1} (a,a') - d^{\G_2} (a,a')| \\
		 & \leq & \frac{r}{2} . \label{eqn:thmgh1-2}
\end{eqnarray}
The result follows from estimates (\ref{eqn:thmgh1-1}) and (\ref{eqn:thmgh1-2}).
\end{proof}

\begin{condition} \label{thm:gh2-hyp}
Given collections of subsets $\G_i$, $i = 1, 2$, of a metric space $(X,d)$, suppose that there exist $g_i \in \G_i$ and $D \subset X$ satisfying:
\begin{enumerate}
 \item \label{thm:gh2-hyp1a}
 For any $x, y \in X \setminus D$, $x \, \G_1 \, y$ if, and only if, $x \, \G_2 \, y$.
 
 \item \label{thm:gh2-hyp2a}
 If $x \in X \setminus D$ and $y \in D$ are such that $x \, \G_i \, y$, then $x,y \in g_i$ ($i = 1, 2$).
 
 \item \label{thm:gh2-hyp3a}
 For $i = 1, 2$, $g_i \cap (\partial D \setminus D) \neq \emptyset$.
\end{enumerate}
\end{condition}

\begin{prop} 
\label{prop:gh2-1}
 Let $X$ be a metric space and $D \subset X$. For any $x \in X \setminus D$, $y \in D$ and $z \in \partial D$:
 \begin{equation}
  d(x,z) \leq d(x,y) + \diam D \, .
 \end{equation}
Here, $\partial D$ and $\diam D$ denotes, respectively, the boundary and the diameter of $D$.
\end{prop}
\begin{proof}
 Let $z_n$ be a sequence in $D$ converging to $z$. For each $n$, the triangle inequality and the definition of diameter gives: 
\[
d(x, z_n) \leq d(x,y) + d(y, z_n) \leq d(x,y) + \diam D,
\]
Then, since $p \mapsto d (x,p)$ is a continuous map $X \to \R$, it follows:
\begin{eqnarray*}
d(x, z) & =  & d(x, \lim z_n) = \lim d(x, z_n) \\
& \leq & \lim (d(x,y) + d(y, z_n)) = d(x,y) + \lim d(y, z_n) \\
& \leq & d(x,y) + \diam D.
\end{eqnarray*}
\end{proof}

\begin{lemma} 
\label{lemma:gh2-2}
For collections of subsets $\G_i$, $i = 1, 2$, of a metric space $(X,d)$, if Conditions \ref{thm:gh2-hyp} are fulfilled, then the associated quotient semi-metrics $d^{\G_1}$ and $d^{\G_2}$ are such that:
\begin{equation}
 |d^{\G_1} (x,y) - d^{\G_2} (x,y)| \leq 2 \diam D \quad \forall x,y \in X \setminus D \, .
\end{equation}
\end{lemma}

\begin{proof}
Denote $E = X \setminus D$.
It will be proven that:
\begin{equation}
 d^{\G_2} (x,y) \leq d^{\G_1} (x,y) + 2 \diam D \quad \forall x,y \in E \, .
\end{equation}
The proof that 
$d^{\G_1} (x,y) \leq d^{\G_2} (x,y) + 2 \diam D$ for any $x,y \in E$
is analogous, implying the result.

Let $x, y \in E$.
Due to Condition \ref{thm:gh2-hyp}-\ref{thm:gh2-hyp1a}, every $\G_1$-walk contained in $E$ is a $\G_2$-walk (of course with the same length).
Then, it suffices to show that to each $\G_1$-walk $\W = \{ x_j, y_j \}_{j=1}^N$ from $x$ to $y$ corresponds a $\G_1$-walk contained in $E$ with length at most $\ell(\W) + 2 \diam D $. 
Suppose that $\W$ intersects $D$, and let $1 \leq j_0 \leq j_1 \leq N$ be the first and last indices such that this happens: $\{ x_j, y_j \} \cap D \neq \emptyset$ for $j = j_0, j_1$; and $\{ x_j, y_j \} \cap D = \emptyset$ for every $j < j_0$ and $j > j_1$. 
The subwalk $\{x_{j_0}, y_{j_0} \, ; \ldots ; \, x_{j_1}, y_{j_1} \}$ will be modified to obtain the result. 
The cases when the walk enters and leaves $D$ through a steps or jumps are treated separately.

\textit{Case 1. $\W$ enters and leaves $D$ through jumps:} $x_{j_0}, y_{j_1} \in D$. In this case, since $y_{j_0 - 1}, x_{j_1 + 1} \in E$, Condition \ref{thm:gh2-hyp}-\ref{thm:gh2-hyp3a} implies that $y_{j_0 - 1}, x_{j_1 + 1} \in g_1$. Then, $\{x_{j_0}, y_{j_0} \, ; \ldots ; \, x_{j_1}, y_{j_1} \}$ can be simply removed from the walk, and the result is a $\G_1$-walk contained in $E$ with smaller length.

\textit{Case 2. $\W$ enters and leaves $D$ through steps:} $x_{j_0}, y_{j_1} \in E$ and, as a consequence, $y_{j_0}, x_{j_1} \in D$. In particular, $j_0 \neq j_1$. Condition \ref{thm:gh2-hyp}-\ref{thm:gh2-hyp3a} says that there exist $y_{j_0} ', x_{j_1} '  \in g_1 \cap (\partial D \setminus D)$, and Proposition \ref{prop:gh2-1} guarantee that:
\begin{equation}
 d(x_{j_0}, y_{j_0} ') + d(x_{j_1} ', y_{j_1}) \leq  d(x_{j_0}, y_{j_0}) + d(x_{j_1}, y_{j_1})  + 2 \diam D.
\end{equation}
It follows that replacing $\{x_{j_0}, y_{j_0} \, ; \ldots ; \, x_{j_1}, y_{j_1} \}$ by $\{ x_{j_0}, y_{j_0} ' \, ; \, x_{j_1} ', y_{j_1} \}$ produces the wanted $\G_1$-walk.
 
\textit{Case 3. $\W$ enters $D$ through a step and leaves it through a jump:} $x_{j_0} \in E$, $y_{j_0} \in D$ and $y_{j_1} \in D$. In particular, $j_0 > 1$ and $j_1 < N$. Then $N \geq j_1 + 1$, $x_{j_1 + 1} \in E$, and Condition \ref{thm:gh2-hyp}-\ref{thm:gh2-hyp2a} implies that $x_{j_1 + 1} \in g_1$. As in the previous case, Condition \ref{thm:gh2-hyp}-\ref{thm:gh2-hyp3a}, gives $y_{j_0} ' \in g_1 \cap (\partial D \setminus D)$, and again Proposition \ref{prop:gh2-1} shows that replacing $\{x_{j_0}, y_{j_0} \, ; \ldots ; \, x_{j_1}, y_{j_1} \}$ by $\{ x_{j_0}, y_{j_0} ' \}$ gives the desired $\G_1$-walk, as :
\begin{equation}
d(x_{j_0}, y_{j_0} ') \leq  d(x_{j_0}, y_{j_0}) + \diam D.
\end{equation}

By symmetry, the case when the walk enters through a jump and leaves through a step is analogous to Case 3.
\end{proof}

% \begin{thm} \label{thm:gh2}
% Given decompositions $\G$ and $\G_n$, $n \in \N$, of a compact metric space $(X,d)$, suppose that there exist $g \in \G$ and, for each $n \in \N$, $g_n \in \G_n$ and $E_n \subset X$ satisfying Conditions \ref{thm:gh2-hyp}. Let $S$ and $S_n$ be the metric quotients associated to $\G$ and $\G_n$, respectively. As $n \to + \infty$, if $\diam D_n \to 0$ , then $S_n \to S$ in the Gromov-Hausdorff sense.
% \end{thm}

\begin{proof}[Proof of Theorem \ref{thm:main}]
For each $n \in \N$, the hypothesis says precisely that $g_n \in \G_n$, $g^n \in \G$ and $D_n \subset X$ fulfills Conditions \ref{thm:gh2-hyp}.
Denote $E_n = X \setminus D_n$, and fix any $\delta_0 > 0$.
Then, Lemma \ref{lemma:gh2-2} gives:
\begin{equation} \label{eqn:mainthm1}
|d^{\G_n} (x,y) - d^\G (x,y)| < 2 \diam D_n < 2 (1+\delta_0) \diam D_n \quad \forall x, y \in E_n \, .
\end{equation}

It's clear that, for each $n \in \N$, $x \in X$ and $r > \diam D_n$, $B(x,r) \cap E_n \neq \emptyset$.
In particular, each $E_n$ is a $[2(1+\delta_0)\diam D_n]$-net on $X$.
Thus, by (\ref{eqn:mainthm1}), the hypothesis of Lemma \ref{thm:gh1} hold for $r = 2(1+\delta_0) \diam D_n$, and it follows that:
\begin{equation}
\dgh (S_n,S) < 5(1 + \delta_0)\diam D_n .
\end{equation}
Finally, given $\varepsilon > 0$, let $n_0 \in \N$ be such that $5 (1+\delta_0) \diam D_n < \varepsilon$ for every $n \geq n_0$, so $\dgh (S_n,S) < \varepsilon$.
\end{proof}

\section{Application to paper-folding schemes}
\label{sec:paperlimit}

The conditions of Theorem \ref{thm:main} can be verified in the context of plain paper-folding schemes, yielding:

\begin{thm}
\label{thm:paperlimit}
 Given a plain paper-folding scheme $(P, \pp)$ and a sequence $\gamma_n \subset \partial P$, $n \in \N$, of $\pp$-plain arcs, let $(P,\pp_n)$ be a sequence of paper-folding schemes such that $\pp_n$ coincides with $\pp$ on $\partial P \setminus \overset{\circ}{\gamma_n}$ and $\gamma_n$ is $\pp_n$-plain for every $n$. 
 Let $S$ and $S_n$ be the paper spaces associated to $(P, \pp)$ and $(P,\pp_n)$, respectively. 
 If $|\gamma_n| \to 0$ as $n \to \infty$, then $\dgh (S, S_n) \to 0$.
\end{thm}
\begin{proof}
 Let $\sim_{\pp}$ and $\sim_{\pp_n}$ be the equivalence relations associated to the quotient semi-metrics $d^\pp$ and $d^{\pp_n}$, respectively, and take as $\G$ and $\G_n$ the collections of $\sim_\pp$ and $\sim_{\pp_n}$-equivalence classes in accordance.
 Notice that $d^\pp = d^{\G}$ and $d^{\pp_n} = d^{\G_n}$, so the paper spaces $S$ and $S_n$ are precisely the metric quotients associated to $\G$ and $\G_n$.
 Let $D_n = \overset{\circ}{\gamma_n}$, and denote by $z_n$ and $w_n$, $z_n \neq w_n$, its endpoints, whose equivalence classes are denoted $[z_n]_{\sim_\pp} = [w_n]_{\sim_\pp}$ and $[z_n]_{\sim_{\pp_n}} = [w_n]_{\sim_{\pp_n}}$.
 These equalities are guaranteed by the hypothesis that each $\gamma_n$ is both $\pp$ and $\pp_n$-plain, due to Theorem \ref{thm:plain}.
 Finally, let $g^n = [z_n]_{\sim_\pp} \in \G$ and $g_n = [z_n]_{\sim_{\pp_n}}  \in \G_n$.
 The conditions of Theorem \ref{thm:main} will be verified in the sequence.
  
 \textit{Condition \ref{thm:gh2-hyp1}.}
 Let $x,y \in P \setminus D_n$ be such that $x \, \G \, y$.
 If $x = y$, then it is obvious that $x \, \G_n \, y$.
 For instance, this is the case when one of these points is interior to $P$.
 So assume that $x \neq y$ and $x, y \in \partial P$.
 In case $\{ x, y \}$ is an interior $\pp$-pair, then it is also an interior $\pp_n$-pair, and $x \, \G_n \, y$, since $\pp_n$ coincides with $\pp$ on $\partial P \setminus D_n$.
 And in case $\{ x, y \}$ is not an interior $\pp$-pair, consider the arc $[x,y]$ having these points as endpoints contained in $\partial P \setminus D_n$.
 Since $(P,\pp)$ is plain, Theorem \ref{thm:plain} applies, and $[x,y]$ is a $\pp$-plain arc. 
 It is also a $\pp_n$-plain arc, since $\pp_n$ coincides with $\pp$ on $[x,y]$. 
 Then, $x \, \G_n \, y$ follows from Theorem \ref{thm:plain}.
 The proof that $x \, \G_n \, y$ implies $x \, \G \, y$ for $x,y \in P \setminus D_n$ is completely analogous, the last step being valid since $(P,\pp_n)$ is plain paper-folding scheme, which can be easily verified.

 \textit{Condition \ref{thm:gh2-hyp2}.}
 Theorem \ref{thm:plain} says that, since each $\gamma_n$ is both $\pp$ and $\pp_n$-plain, every $\gamma_n \setminus g^n$ and $\gamma_n \setminus g_n$ are, respectively, $\sim_{\pp}$ and $\sim_{\pp_n}$-saturated.
 By definition, this means that $[y]_{\sim_{\pp}} \subset \gamma_n \setminus g^n$ (resp. $[y]_{\sim_{\pp_n} } \subset \gamma_n \setminus g_n$) for every $y \in \gamma_n \setminus g^n$ (resp. $y \in \gamma_n \setminus g_n$).
 Of course, for each $y \in D_n$, either $y \in g^n$ (resp. $y \in g_n$), or $y \in \gamma_n \setminus g^n$ (resp. $y \in \gamma_n \setminus g_n$).
 Therefore, if $x \, \G \, y$ (resp. $x \, \G_n \, y$) with $x \in P \setminus D_n$ and $y \in D_n$, then $x,y \in g^n$ (resp. $x, y \in g_n$).

 \textit{Condition \ref{thm:gh2-hyp3}.} 
 As $\partial D_n = \gamma_n$, $\partial D_n \setminus D_n = \{ z_n, w_n \}$, so it's clear that $g^n \cap (\partial D_n \setminus D_n) \neq \emptyset$ and $g_n \cap (\partial D_n \setminus D_n) \neq \emptyset$. 

 The proof is concluded by noticing that $\lim_{n \to \infty}\diam D_n = 0$ is a consequence of $\lim_{n \to \infty} |\gamma_n| = 0$. This is imediate if $\gamma_n$ does not contain any point whose angle internal to $P$ is smaller than $\pi$, as in this case $\diam D_n = |\gamma_n|$; and follows from the Law of Cosines, otherwise.  
 \end{proof}

 The following examples illustrates how to use Theorem \ref{thm:paperlimit}. 
 It includes every identification pattern that appears in \cite{unimodal}. 
 Recall that in the associated Figures, dotted lines connect points paired by the paper-folding scheme.
 
 \begin{figure}
\centering
\includegraphics{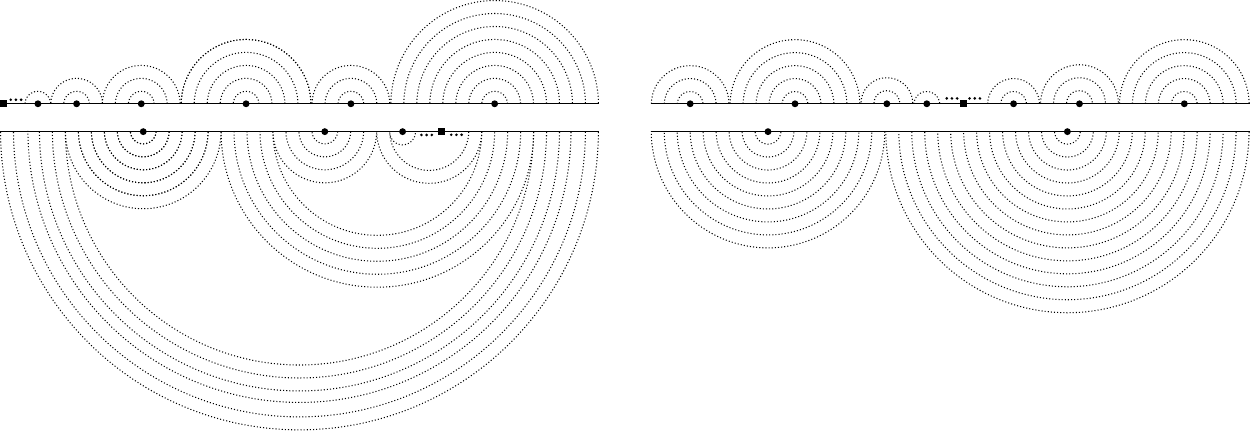}
\caption{Identification patterns (\ref{eqn:canon2}), (\ref{eqn:canon1}) and (\ref{eqn:singular1}) on the arcs $\gamma_n$ in Example \ref{exem:general}, and a couple of folds used for approximating them. 
The point $p$ corresponds to $\blacksquare$.}
\label{fig:general}
\end{figure} 
  
\begin{exem}[Figure \ref{fig:general}]
\label{exem:general}
Suppose that a plain paper-folding scheme $(P,\pp)$ is such that a point $p \in \partial P$ is the intersection of a nested sequence $\gamma_n$, $n \in \N^*$, of arcs contained in $\partial P$, on which $\pp$ has one of the following forms:
 \begin{equation} \label{eqn:canon2}
  p \cdots \alpha_{n+1} ' \, \alpha_{n+1} \, \alpha_n ' \, \alpha_n \, .
 \end{equation}
 \begin{equation} \label{eqn:canon1}
 \alpha_{-n} \, \alpha_{-n} ' \, \alpha_{-n-1} \, \alpha_{-n-1} ' \cdots p \cdots \alpha_{n+1} ' \, \alpha_{n+1} \, \alpha_n ' \, \alpha_n \, .
 \end{equation}
\begin{equation} \label{eqn:singular1}
 \beta_n \, \alpha_{n+1} \, \alpha_{n+1} ' \, \beta_{n+1} \, \alpha_{n+2} \, \alpha_{n+2} '  \cdots p \cdots \, \beta_{n+1} ' \, \beta_n ' .
\end{equation}
In case $p$ is a vertex of $P$, assume that every $\gamma_n$ is contained in the pair of sides of $P$ meeting at $P$ (for pattern (\ref{eqn:canon2}) it end up contained in just one of them).

To define $\pp_n$ as in Theorem \ref{thm:paperlimit}, the simplest choices of plain patterns to place on $\gamma_n$ are a couple of folds folds. 
So, for each $n \in \N^*$, let the restriction of $\pp_n$ to $\gamma_n$ be of the form $\phi_n \, \phi_n ' \, \psi_n ' \, \psi_n$.
When $\gamma_n$ is contained in one side of $P$, a single fold can be used as well.
Theorem \ref{thm:paperlimit} readily aplies, and the paper spaces $S$ associated to $(P,\pp)$ are the Gromov-Hausdorff limits of the sequences of paper spaces $S_n$ associated to $(P,\pp_n)$.
Notice that the lengths of the pairings are not important here.
Recall that, due to Theorem \ref{thm:plain}, all these quotients are homeomorphic to $2$-spheres.
\end{exem}

\begin{figure}
\centering
\includegraphics[width=\linewidth]{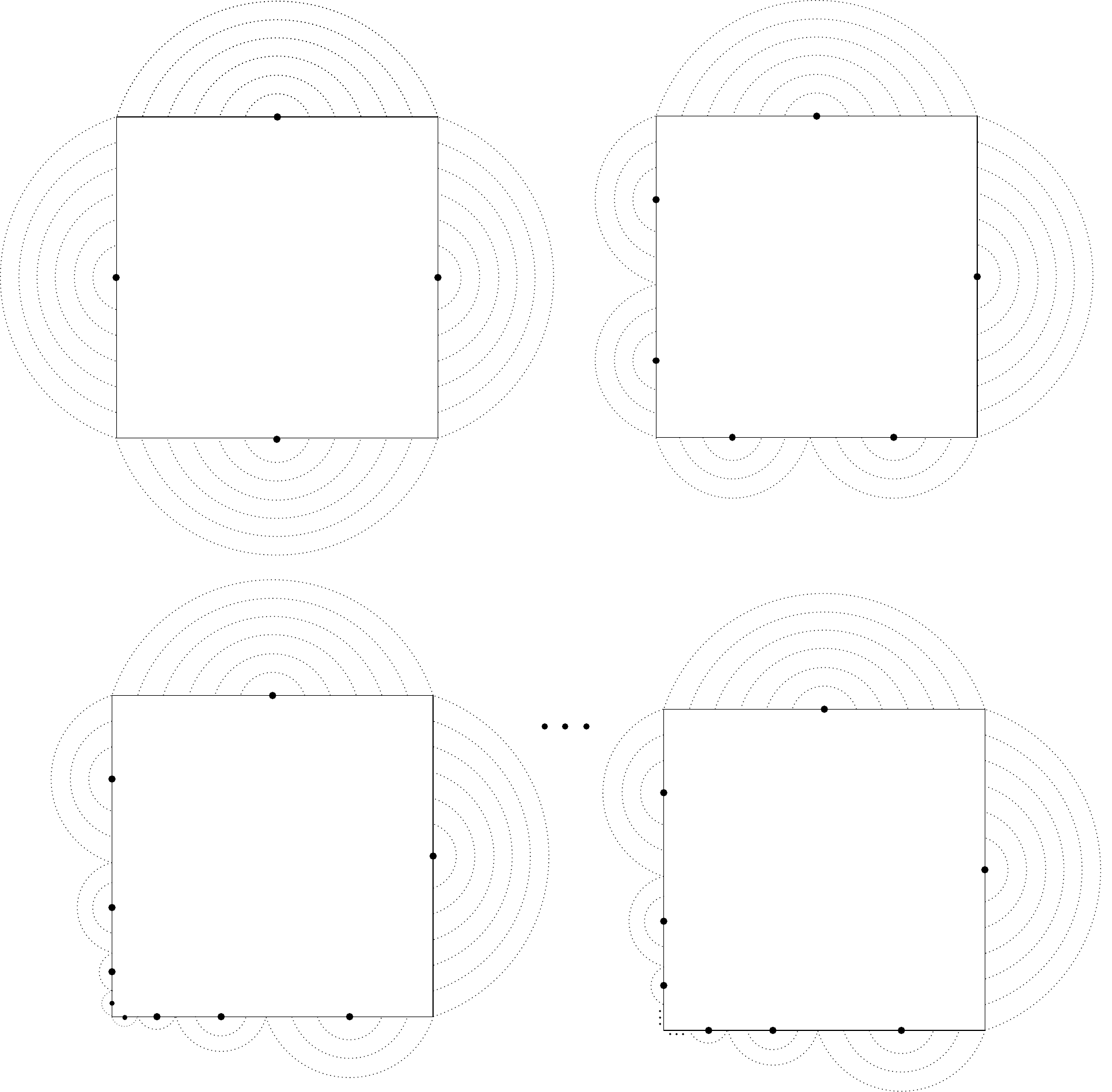}
\caption{Sequences of conic-flat spheres without singularities converging to conic-flat spheres with precisely one singularity (Example \ref{exem:canon1}).}
\label{fig:canon1}
\end{figure}

\begin{exem}[Figure \ref{fig:canon1}]
 \label{exem:canon1}
 As a particular case of Example \ref{exem:general}, Figure \ref{fig:canon1} shows plain a paper-folding scheme on a square that is of the form (\ref{eqn:canon1}) around the inferior left-hand vertex $p$, and three elements of the sequence constructed above approximating it.
When, as indicated, $P$ is an Euclidean square and the lengths of folds producing the limiting paper space are chosen by halving, $S$ is the domain of the self-homeomorphism of the sphere known as the tight horseshoe \cite{dCH}.

Independently of these choices, Figure \ref{fig:canon1} gives concrete examples of sequences of conic-flat spheres without singularities converging to a conic-flat sphere with precisely one singular point, the projection $\hat{p}$ of $p$. 
It possesses the following two singular properties. 
First, $\hat{p}$ is the accumulation of the projections of the folding points marked as $\bullet$ in Figure \ref{fig:canon1}. 
These are conical points with total angle around it in $S$ equal to $\pi$ and, thus, with positive curvature. 
Second, $\hat{p}$ is \textit{the vertex of an $\infty$-od in $S$}: there exist a convex subspace $K$ of $S$, whose points are flat except for $\hat{p}$, isometric to the intrinsic metric of a countable infinite collection of pairwise disjoint half-closed intervals glued together by its endpoints, which corresponds to $\hat{p}$ (here, convex means that every shortest path connecting points in $K$ is contained in $K$).
This is related to $\hat{p}$ having infinite total angle around it in $S$.
Due to these two singular properties, arbitrarily small neighborhoods of $\hat{p}$ contain geodesic triangles implying that $S$ violates the definitions of bounded curvature, both above and below, in the sense of comparative geometry \cite{BBI}.
Analogous properties hold for the projection of the point $p$ in pattern (\ref{eqn:canon2}).
\end{exem}

\begin{remark}
The projection of the point $p$ in pattern (\ref{eqn:singular1}) is in $S$ the accumulation of conical points with total angles equal to $\pi$ and $3 \pi$ around it, while it is not the vertex of an $\infty$-od in $S$.
Singular properties such as this and the ones mentioned above will be rigorously pursued in further works, as mentioned in section \ref{sec:back-paperspaces}.  
\end{remark}

\begin{figure}
\centering
\includegraphics{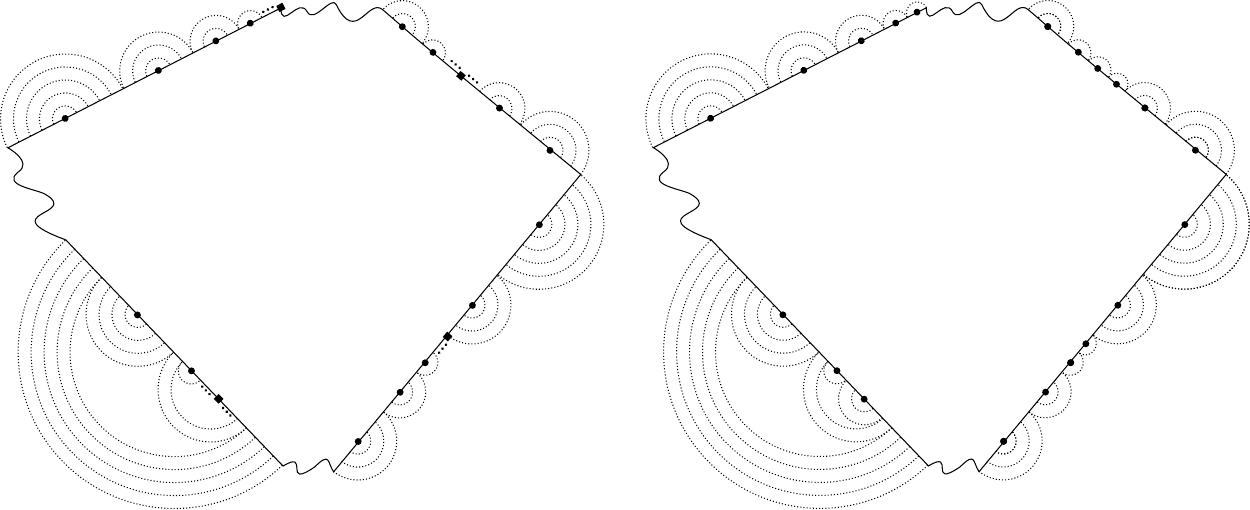}
\caption{Applying the construction in Example \ref{exem:general} repeatedly (Example \ref{exem:repeat}). Loose lines represent generic portions of $\partial P$.}
\label{fig:repeat}
\end{figure}

\begin{exem}[Figure \ref{fig:repeat}]
 \label{exem:repeat}
 As an example of how a given plain paper-folding scheme $(P, \pp)$ can be successively ``simplified'' in order to approximate its quotient $S$ by a sequence of spheres with fewer / less complicated singularities, suppose that there are pairwise distinct points $p_1, \ldots, p_m \in \partial P$ around which $\pp$ has the forms of Example \ref{exem:general}.
 Given $\varepsilon > 0$, let $n_1 \in \N$ be such that $\dgh (S, S_{1,n}) < \varepsilon/m$ for every $n \geq n_1$, where $S_{1,n}$ is the paper space associated to a $\pp_{1,n}$ constructed as in Example \ref{exem:general}.
 Now, approximate $S_{1,n_1}$ by modifying $\pp_{1,n_1}$ around $p_2$, obtainining a sequence $\pp_{2,n}$ and $n_2 \in \N$ whose quotients $S_{2,n}$ satisfy $\dgh (S_{1,n_1},S_{2,n}) < \varepsilon/m$ for every $n \geq n_2$.
 Repeating this for $S_{2,n_2}$, and so on, a sequence $S_{m,n}$ converging to $S_{m-1,n_{m-1}}$ as $n \to \infty$ is obtained, associated to patterns $\pp_{m,n}$ that coincides with the original $\pp$ except at small arcs containing $p_1, \ldots, p_m$, where it contains only folds. 
 Due to the triangle inequality for $\dgh$, $\dgh(S,S_{m,n}) < \varepsilon$ for every $n \geq n_m$.
\end{exem}

% \begin{figure}
% \centering
% \def\svgwidth{\columnwidth}
% \input{canon1dir1.pdf_tex}
% \caption{A limiting sphere.}
% \label{fig:canon1convgh}
% \end{figure}

% \begin{corol}
%  The paper space of Figure \ref{fig:canon1} is the Gromov-Hausdorff limit of the sequence of paper spaces of Example 	\ref{exem:canon1-approx}.
% \end{corol}
% \begin{proof}
% On Theorem \ref{thm:gh2}, take $D_n = \intt_{\partial P} (\beta_n \cup \beta_n ')$, so $E_n = P \setminus D_n$; $g \in \G$ and $g_n \in \G_n$ the elements that contain the inferior left vertex of $P$. Conditions \ref{thm:gh2-hyp} holds and $\diam D_n = 2 \sqrt{2} b_n \to 0$ as $n \to + \infty$, so the result follows.
% \end{proof}

\bibliographystyle{alpha}
\bibliography{references}

\end{document}